\documentclass{amsart}
\usepackage[T2A,T1]{fontenc}
\usepackage[utf8]{inputenc}
\usepackage[russian,USenglish]{babel}
\usepackage{amsfonts}
\usepackage{amssymb}
\usepackage{amsmath}
\usepackage{amsthm}
\usepackage{color}
\usepackage{xcolor}

\usepackage{enumitem,moreenum}
\usepackage[bookmarks=true,plainpages=false,linktocpage,colorlinks=true,hidelinks,citecolor=green!80!black,linkcolor=red!70!black,filecolor=magenta,urlcolor=magenta,breaklinks,pdfauthor={Bernardo González Merino, Thomas Jahn, Christian Richter},pdftitle={Uniqueness of circumcenters in generalized Minkowski spaces}]{hyperref}

\newcommand{\rus}[1]{\selectlanguage{russian}{\fontfamily{Tempora-TLF}\fontsize{9pt}{11pt}\selectfont #1}\selectlanguage{USenglish}}

\makeatletter
\renewcommand*{\eqref}[1]{%
  \hyperref[{#1}]{\textup{\tagform@{\ref*{#1}}}}%
}
\makeatother

\def\inter{\mathop\mathrm{int}\nolimits}
\def\cl{\mathop\mathrm{cl}\nolimits}
\def\relint{\mathop\mathrm{relint}\nolimits}
\def\bd{\mathop\mathrm{bd}\nolimits}
\def\cc{\mathop\mathrm{cc}\nolimits}
\def\conv{\mathop\mathrm{conv}\nolimits}
\def\lin{\mathop\mathrm{lin}\nolimits}
\def\R{\mathbb{R}}

\theoremstyle{plain}
\newtheorem{theorem}{Theorem}
\newtheorem{lemma}[theorem]{Lemma}
\newtheorem{corollary}[theorem]{Corollary}

\theoremstyle{definition}
\newtheorem{example}[theorem]{Example}

\theoremstyle{remark}
\newtheorem{remark}[theorem]{Remark}
\newtheorem*{myproof}{Proof of \hyperref[thm:UniqueCircumGauge]{Theorem~\ref*{thm:UniqueCircumGauge}}}

\begin{document}

	
	\title[Uniqueness of circumcenters in generalized Minkowski spaces]{Uniqueness of circumcenters in generalized Minkowski spaces}
	
	\author[B. Gonz\'alez Merino]{Bernardo Gonz\'{a}lez Merino}
	\address{Departamento de Analisis Matemático, Facultad de Matem\'aticas, Universidad de Sevilla, Apdo. 1160, 41080-Sevilla, Spain}
	\email{bgonzalez4@us.es}

	\author[T. Jahn]{Thomas Jahn}
	\address{Faculty of Mathematics, University of Technology, 09107 Chemnitz, Germany}
	\email{thomas.jahn@mathematik.tu-chemnitz.de}
		
        \author[C. Richter]{Christian Richter}
        \address{Institute of Mathematics, Friedrich Schiller University, 07737 Jena, Germany}
        \email{christian.richter@uni-jena.de}
	
	
	\thanks{The first named author was partially supported by
		Fundaci\'{o}n S\'{e}neca, Programme in Support of Excellence Groups of the Regi\'{o}n de Murcia, Spain, project reference 19901/GERM/15, and by
MINECO project reference MTM2015-63699-P, Spain.}
	
	\date{\today}
	\begin{abstract}
	  In an $n$-dimensional normed space every bounded set has a unique circumball if and only if every set of cardinality two has a unique circumball and if and only if the unit ball of the space is strictly convex. When the symmetry of the norm is dropped, i.e., when the centrally symmetric unit ball is replaced by an arbitrary convex body,
then the above three conditions are no longer equivalent. We show for the latter case that every bounded set has a unique circumball if and only if every set of cardinality at most $n$ has a unique circumball. We also give an equivalent condition in terms of the geometry of the unit ball.
In similar terms we answer the following more general question for every $k \in \{0,\ldots,n-2\}$: When are the dimensions of the sets of all circumcenters of arbitrary bounded sets not larger than $k$?
	\end{abstract}
	\maketitle

		
\section{Introduction}

The interest in finding the balls of minimal radius containing a
given set (its circumballs) has been vocalized for the first time by Sylvester \cite{Sylvester1857}.
Since then, the problem received attention from several mathematical communities, resulting in various names under which the problem is known (center problem, minimal enclosing ball problem, minimax location problem, Sylvester problem, optimal containment problem).
The attention has not been restricted to the solution of
the original problem which Sylvester posed for the Euclidean plane, but has also yielded extensions of the problem obtained by transferring conceptual details of the problem to other contexts in order to apply specific methods.
The centers of the circumballs of a given set (its circumcenters) serve as approximations of the latter in the sense that they are uniformly close to all of its points. Therefore it is promising to investigate the geometry of the set of circumcenters.

In his paper \cite{Zindler1920}, Zindler proves not only the uniqueness of the circumcenter in two-dimensional and three-dimensional Euclidean space.
Zindler also shows that the points where the given set touches the boundary of its circumball are well spread in the sense that they
do not lie in a hemisphere.
Finally, Zindler proposes in \cite{Zindler1920} the study of an analogous problem in multiple dimensions where Euclidean balls have been replaced by homothetic copies of an arbitrary but fixed convex body $C$.
As a first step to the full generality of this extension, it has been investigated in normed spaces, i.e. $C=-C$, as early as 1962 \cite{Garkavi1962}.
Variants of Sylvester's problem and Zindler's extension also appear in approximation theory \cite{AmirZi1980}, location science \cite{ElzingaHe1972a}, computational geometry \cite[Theorem~14]{ShamosHo1975}, and convex analysis \cite{NamHo2013}. An elementary account of the location of circumcenters of triangles in normed planes is \cite{AlonsoMaSp2012b}.
Termed optimal containment problem, the generalized problem has been addressed in \cite{BrandenbergRo2011,EavesFr1982},
but also in \cite{BrandenbergKoe2013} where one can find a corresponding result on touching points being well spread.

The starting point of our work is the following characterization of normed spaces that contain subsets having more than one circumcenter.
\begin{theorem}[{\cite[Lemma~8]{JahnMaRi2017}}, see also {\cite[Theorem~VI]{Garkavi1962}}]\label{thm:UniqueCircumSymmetric}
Let $(\R^n,\|\cdot\|)$ be a normed space with unit ball $B$. Then the following are equivalent:
\begin{enumerate}[label={(\alph*{$_0$})},leftmargin=*,align=left,noitemsep]
\item{There exists a convex  body $K \subseteq \R^n$ not having a unique circumcenter.\label{nonunique-symmetric-body}}
\item{There exists a set $\{x_1,x_2\} \subseteq \R^n$ not having a unique circumcenter.\label{nonunique-symmetric-2}}
\item{The ball $B$ is not strictly convex; that is, the boundary of $B$ contains a non-degenerate line segment.\label{existence-segment}}
\end{enumerate}
\end{theorem}
Here \ref{nonunique-symmetric-2} says that it suffices to have uniqueness for sets of two points to get uniqueness for all bodies. Condition \ref{existence-segment} is a simple geometric property of the unit ball.

We generalize the situation in two ways. On the one hand, we drop the symmetry condition of the unit ball $B$. That is, we replace $B$ by an arbitrary full-dimensional convex body $C$ and ask for covers of bounded sets by smallest possible homothetical images of $C$, as originally proposed in \cite{Zindler1920}. On the other hand, note that condition \ref{nonunique-symmetric-body} claims the existence of some body $K$ whose set of circumcenters has a dimension larger than $0$. We shall characterize, more generally, the situation where this dimension exceeds $k$. The above mentioned result of Brandenberg and K\"onig \cite[Theorem~2.3]{BrandenbergKoe2013} will serve for linking the affine dimension of the set of circumcenters to the boundary structure of $C$.

We use the following notations.
Let $\mathcal{K}^n= \{K \subseteq \R^n: K$ is compact, convex and non-empty$\}$ denote the family of \emph{convex bodies} in $\R^n$ (where compactness is understood in the canonical topology of the linear space $\R^n$). Every $B \in \mathcal{K}^n$ that is symmetric with respect to the origin $0$, i.e., $B=-B$, and has a non-empty interior defines a norm $\|x\|= \min\{\lambda \ge 0: x \in \lambda B\}$ on $\R^n$. We speak of a \emph{Minkowski space} (or finite-dimensional real normed space) $(\R^n,\|\cdot\|)$ with \emph{unit ball} $B$. When replacing $B$ by some $C \in \mathcal{K}^n$ that contains $0$ in its interior $\inter(C)$, we obtain a \emph{gauge} $\gamma(x)=\min\{\lambda \ge 0: x \in \lambda C\}$ on $\R^n$. In general, gauges are not symmetric in the sense that $\gamma(x)=\gamma(-x)$. We call $(\R^n,\gamma)$ a \emph{generalized Minkowski space} with \emph{unit ball} (or \emph{gauge body}) $C$ (cf., e.g., \cite[p.~4]{Zalinescu2002}).

The \emph{circumradius} of a bounded non-empty set $A \subseteq \R^n$ with respect to $C$ is
$$
R(A,C)=\inf\{\varrho \ge 0: \text{ there exists } x \in \R^n \text{ such that } A \subseteq \varrho C+x\}.
$$
When $A \subseteq R(A,C)C+x$, we call $R(A,C)C+x$ a \emph{circumball} and $x$ a \emph{circumcenter} or \emph{Chebyshev center} \cite{Garkavi1962} of $A$. Clearly, $R(A,C)=R(\cl(A),C)=R(\conv(A),C)$, where $\cl(A)$ and $\conv(A)$ denote the closure and the convex hull of $A$, respectively. Therefore it suffices to study circumcenters of convex bodies. The set of all circumcenters
$$
\cc(A,C)=\{x \in \R^n: A \subseteq R(A,C)C+x\}
$$
is sometimes called the \emph{Chebyshev set} of $A$ \cite{MMS}.
This set is known to be a convex body whose dimension does not exceed $n-1$ \cite[Corollary~4.5 and Theorem~4.7]{Jahn2017}.

We will use $\langle \cdot,\cdot \rangle$ for denoting the standard inner product in $\R^n$. The \emph{support function} of $K \in \mathcal K^n$ for $u \in \R^n$ is denoted by $h(K,u)=\max\{\langle x,u \rangle: x \in K\}$.
A \emph{support set} of $K\in\mathcal K^n$ is a set $A\subseteq\bd(K)$ such that $A=\{x\in K:h(K,u)=\langle x,u\rangle\}$ for some $u\in\R^n\setminus\{0\}$.

The \emph{normal cone} of $K$ at a point $x$ from the boundary $\bd(K)$ of $K$ is $N(K,x)=\{u \in \R^n: h(K,u)=\langle x,u \rangle\}$; that is, $N(K,x) \setminus \{0\}$ consists of all outer normal vectors of $K$ at $x$. The \emph{relative interior} $\relint(K)$ of $K \in \mathcal K^n$ is the interior of $K$ with respect to the canonical topology of the affine span of $K$.

Let $\lin(A)$ denote the \emph{linear span} of $A \subseteq \R^n$. The linear subspace orthogonal to $A$ is given by
$$
A^\perp=\{x \in \R^n: \langle a,x \rangle=0 \mbox{ for all } a \in A\}.
$$
We write $\dim(A)$ for the \emph{dimension} of the affine span of $A$.
The \emph{line segment} with endpoints $v,w \in \R^n$ is $[v,w]=\conv\{v,w\}$. The $i$th \emph{unit coordinate vector} is denoted by $e_i=(0,\ldots,0,\stackrel{i}{1},0,\ldots,0) \in \R^n$.


\section{Main result}

\begin{theorem}\label{thm:UniqueCircumGauge}
Let $(\R^n,\gamma)$ be a generalized Minkowski space with unit ball $C$ and let $k \in \{0,\ldots,n-2\}$. Then the following are equivalent:
\begin{enumerate}[label={(\alph*{$_k$})},leftmargin=*,align=left,noitemsep]
\item{There exists $K \in \mathcal{K}^n$ such that $\dim(\cc(K,C)) > k$.\label{dim-k-body-asymmetric}}\renewcommand{\labelenumi}{(\alph{enumi}$_k^\prime$)}
\item{There exist $x_1,\ldots,x_{n-k} \in \R^n$ such that $\dim\left(\cc(\{x_1,\ldots,x_{n-k}\},C)\right) > k$.\label{dim-k-asymmetric-n}}
\item{There exist $x_1,\ldots,x_{n-k} \in \bd(C)$ and a body
$A_k \in \mathcal{K}^n$ with $\dim(A_k) > k$ such that
\begin{itemize}[label={\textbullet},leftmargin=*,align=left,noitemsep]
\item{$x_i+A_k \subseteq \bd(C)$ for $i=1,\ldots,{n-k}$ and}
\item{there are outer normal vectors $u_i \in N(C,x_i) \setminus \{0\}$, $i=1,\ldots,{n-k}$, such that $0 \in \conv\{u_1,\ldots,u_{n-k}\}$.}
\end{itemize}
\label{existence-dim-k-asymmetric}}
\end{enumerate}
(The vectors $x_1,\ldots,x_{n-k}$ are not required to be mutually different.)
\end{theorem}

We shall use the following tool.
\begin{lemma}[{\cite[Theorem~2.3]{BrandenbergKoe2013}}]\label{opt}
Let $C\in\mathcal K^n$ be the unit ball of a generalized Minkowski space $(\R^n,\gamma)$ and let $K\in\mathcal K^n$ be such that $K\subseteq C$. The following are equivalent:
\begin{itemize}[label={\textbullet},leftmargin=*,align=left,noitemsep]
\item{$R(K,C)=1$.}
\item{There exist $x_1,\ldots,x_l\in K\cap\bd(C)$ for some $2\leq l\leq n+1$
and outer normals $u_i\in N(C,x_i)\setminus\{0\}$, $i=1,\ldots,l$, such that $0\in\conv\{u_1,\ldots,u_l\}$.}
\end{itemize}
\end{lemma}

\begin{myproof}
The implication \ref{dim-k-asymmetric-n}$\Rightarrow$\ref{dim-k-body-asymmetric} is obvious.

To see \ref{existence-dim-k-asymmetric}$\Rightarrow$\ref{dim-k-asymmetric-n}, we use the points $x_1,\ldots,x_{n-k}$ from \ref{existence-dim-k-asymmetric}. By \hyperref[opt]{Lemma~\ref*{opt}}, the second part of \ref{existence-dim-k-asymmetric} shows that $R(\{x_1,\ldots,x_{n-k}\},C)=R(\conv\{x_1,\ldots,x_{n-k}\},C)=1$. The first part of \ref{existence-dim-k-asymmetric} gives $\{x_1,\ldots,x_{n-k}\}+A_k \subseteq C$; i.e., $\{x_1,\ldots,x_{n-k}\}\subseteq C-v$ for all $v \in A_k$. So $-A_k \subseteq \cc(\{x_1,\ldots,x_{n-k}\},C)$ and $\dim\left(\cc(\{x_1,\ldots,x_{n-k}\},C)\right) > k$.

For proving \ref{dim-k-body-asymmetric}$\Rightarrow$\ref{existence-dim-k-asymmetric}, we can suppose that $R(K,C)=1$ and $K \subseteq C$.
By \hyperref[opt]{Lemma~\ref*{opt}}, there are
$x_i \in K \cap \bd(C)$ and $u_i \in N(C,x_i) \setminus \{0\}$, $i=1,\ldots,n+1$, such that $0 \in \conv\{u_1,\ldots,u_{n+1}\}$. So
\begin{equation}\label{eq1}
0= \sum_{i=1}^{n+1} \lambda_i u_i
\end{equation}
for suitable $\lambda_i \ge 0$. Moreover, we can suppose that
\begin{equation}\label{eq2}
\lambda_i > 0, \qquad i=1,\ldots,n+1,
\end{equation}
since we can replace all $x_i$ with $\lambda_i=0$ by some $x_j$ with $\lambda_j > 0$ and split $\lambda_j$ accordingly.

We put $L=\lin \{u_1,\ldots,u_{n+1}\}$. Let $v \in \cc(K,C)$; i.e., $K \subseteq C+v$.
For every $i \in \{1,\ldots,n+1\}$, we have $x_i-v \in K-v \subseteq C$.
This yields $\langle x_i-v,u_i \rangle \le h(C,u_i)=\langle x_i,u_i \rangle$, because $u_i \in N(C,x_i)$. So $\langle v,u_i \rangle \ge 0$ for $i=1,\ldots,n+1$. This together with
$$
0=\langle v,0 \rangle\stackrel{\text{\eqref{eq1}}}{=}\sum_{i=1}^{n+1} \lambda_i \langle v,u_i \rangle
$$
and \eqref{eq2} implies
\begin{equation}\label{eq3}
\langle v,u_i \rangle = 0 \quad\text{for}\quad i=1,\ldots,n+1.
\end{equation}
Consequently, $L \subseteq (\cc(K,C))^\perp$ and $\dim(L) \le n- \dim(\cc(K,C)) < n-k$.

Now Carath\'eodory's theorem shows that the condition $0 \in \conv\{u_1,\ldots,u_{n+1}\}$ can be strengthened to $0 \in \conv\{u_1,\ldots,u_{n-k}\}$ (when $x_1,\ldots,x_{n+1}$ are reordered accordingly), and the second claim of \ref{existence-dim-k-asymmetric} is verified.

Finally, to verify that the first claim of \ref{existence-dim-k-asymmetric} is satisfied with $A_k=-\cc(K,C)$, we shall show that $x_i-v \in \bd(C)$ for all $i \in \{1,\ldots,n-k\}$ and $v \in \cc(K,C)$. We have
$$
x_i-v \in K-v \subseteq (C+v)-v=C.
$$
By $u_i \in N(C,x_i)$, we obtain $h(C,u_i)=\langle x_i,u_i \rangle$. Thus,
$$
\langle x_i-v,u_i \rangle \stackrel{\text{\eqref{eq3}}}{=} \langle x_i,u_i \rangle=h(C,u_i).
$$
The observations $x_i- v \in C$ and $\langle x_i- v,u_i \rangle=h(C,u_i)$ show that $x_i- v \in \bd(C)$. This completes the proof.
\end{myproof}

\begin{remark}
\begin{enumerate}[label={(\roman*)},leftmargin=*,align=left,noitemsep]
\item{A weaker version of condition \ref{dim-k-asymmetric-n} with $\{x_1,\ldots,x_{n-k}\}$ replaced by $\{x_1,\ldots,x_{n+1}\}$ was known to be equivalent to \ref{dim-k-body-asymmetric}, due to \hyperref[opt]{Lemma~\ref*{opt}}.}
\item{Note that the number of points in condition \ref{dim-k-asymmetric-n} of
\hyperref[thm:UniqueCircumGauge]{Theorem~\ref*{thm:UniqueCircumGauge}} is best possible. The following example gives, for all $n \ge 2$ and $0 \le k \le n-2$, an $n$-dimensional generalized Minkowski space such that every set of cardinality at most $n-k-1$ has a unique circumcenter, whereas there exist sets of $n-k$ points having a Chebyshev set of dimension larger than $k$.}
\end{enumerate}
\end{remark}

\begin{example}\label{Ex:Sharpness_of_new_theorem}
Suppose that $n \ge 2$ and $0 \le k \le n-2$. We consider a regular (in the Euclidean sense) $(n-k-1)$-dimensional simplex $\triangle_{n-k-1} \subseteq \lin\{e_1,\ldots,e_{n-k-1}\}$
with vertices $x_1,\ldots,x_{n-k}$ and center of gravity at the origin $0$, and the $(k+1)$-dimensional cube $\Box_{k+1} = [-e_{n-k},e_{n-k}] \times \ldots \times [-e_n,e_n]$. Then we pick a convex body $C\in\mathcal K^n$ with
\begin{equation}\label{eq3'}
\triangle_{n-k-1}+\Box_{k+1} \subseteq C \subseteq (n-k-1)(-\triangle_{n-k-1})+2\Box_{k+1}
\end{equation}
such that
\begin{enumerate}[label={(\greek*)},leftmargin=*,align=left,noitemsep]
\item{$C$ is smooth (i.e., in every $x \in \bd(C)$ there is only one tangent hyperplane),\label{C-smooth}}
\item{all points of $\bd(C) \setminus \bigcup_{i=1}^{n-k} (x_i+\Box_{k+1})$ are exposed,\label{C-exposed}}
\end{enumerate}
see \cite{Gh} for a justification of the existence of $C$. Note that properties \eqref{eq3'}, \ref{C-smooth}, and \ref{C-exposed} imply
\begin{enumerate}[label={(\greek*$^\prime$)},leftmargin=*,align=left,noitemsep]
\item{$N(C,x_i+v)=\{\lambda x_i: \lambda \ge 0\}$ for all $i=1,\ldots,n-k$ and $v \in \Box_{k+1}$,\label{C-normal}}
\item{$x_i+\Box_{k+1}$, $i=1,\ldots,n-k$, are the only support sets of $C$
that are not singletons.\label{C-support}}
\end{enumerate}

We see that $C$ satisfies \ref{existence-dim-k-asymmetric} from \hyperref[thm:UniqueCircumGauge]{Theorem~\ref*{thm:UniqueCircumGauge}} with $A_k=\Box_{k+1}$, because $0=\sum_{i=1}^{n-k} \frac{1}{n-k} x_i$. So there are sets $X$ of cardinality $n-k$, such as $X=\{x_1,\ldots,x_{n-k}\}$, satisfying $\dim(\cc(X,C))>k$.

Now we show that every set of cardinality at most $n-k-1$ has a unique circumcenter. Assume that this is not the case. Then there exist $Y=\{y_1,\ldots,y_{n-k-1}\}$ and $t \in \R^n\setminus \{0\}$ such that
\begin{equation}\label{eq4}
R(Y,C)=1 \qquad\text{and}\qquad Y+\mu t \subseteq C \text{ for all } \mu \in [-1,1].
\end{equation}
Applying \hyperref[opt]{Lemma~\ref*{opt}} to the situation $\mu=0$, we obtain, say, $y_i \in Y \cap \bd(C)$, $1 \le i \le l$, and $u_i \in N(C,y_i) \setminus \{0\}$ such that $0 \in \conv\{u_1,\ldots,u_l\}$. The inclusion $y_i \in Y \cap \bd(C)$ together with property \eqref{eq4} show that $[y_i-t,y_i+t] \subseteq \bd(C)$ for $i=1,\ldots,l$. By \ref{C-support}, the points $y_1,\ldots,y_l$ belong to the $n-k$ cubes mentioned in \ref{C-support}. Now \ref{C-normal} tells us that the normal vectors $u_1,\ldots,u_l$ are positive multiples of not more than $l \le n-k-1$ of the vectors $x_1,\ldots,x_{n-k}$. This contradicts $0 \in \conv\{u_1,\ldots,u_l\}$.
\end{example}


\section{Further observations}

We obtain the following analogue of \hyperref[thm:UniqueCircumGauge]{Theorem~\ref*{thm:UniqueCircumGauge}} for normed spaces.
\begin{theorem}\label{thm:DimNorm}
Let $(\R^n,\|\cdot\|)$, $n \ge 2$, be a Minkowski space with unit ball $B$ and let $k \in \{0,\ldots,n-2\}$. Then the following are equivalent:
\begin{enumerate}[label={(\alph*{$_k$})},leftmargin=*,align=left,noitemsep]
\item{There exists $K \in \mathcal{K}^n$ such that $\dim(\cc(K,B)) > k$.\label{dim-k-symmetric-body}}
\item{There exists a set $\{x_1,x_2\} \subseteq \R^n$ such that $\dim\left(\cc(\{x_1,x_2\},B)\right) > k$.\label{dim-k-symmetric-2}}
\item{The ball $B$ has a support set of dimension larger than $k$.\label{existence-dim-k-symmetric}}
\end{enumerate}
\end{theorem}

\begin{proof}
\hyperref[thm:UniqueCircumGauge]{Theorem~\ref*{thm:UniqueCircumGauge}} gives \ref{dim-k-symmetric-body}$\Rightarrow$\ref{existence-dim-k-asymmetric} and the implications
\ref{dim-k-symmetric-2}$\Rightarrow$\ref{dim-k-symmetric-body} and \ref{existence-dim-k-asymmetric}$\Rightarrow$\ref{existence-dim-k-symmetric} are obvious. It suffices to verify \ref{existence-dim-k-symmetric}$\Rightarrow$\ref{dim-k-symmetric-2}.

Let $S \subseteq \bd(B)$ be a support set of $B$ with $\dim(S) > k$ and let $x_S \in \relint(S)$. Clearly, $R(\{x_S,-x_S\},B)=1$, because $\|x_S-(-x_S)\|=2\|x_S\|=2$ and $\{x_S,-x_S\}\subseteq B$.
Since $x_S \in \relint(S)$, the set $A=(S-x_S) \cap (-S+x_S)$ contains $0$ in its relative interior and satisfies $\dim(A)=\dim(S)>k$. For every $v \in A$,
$$
x_S-v \in x_S-A \subseteq x_S-(-S+x_S)=S \subseteq B
$$
and
$$
-x_S-v \in -x_S-A \subseteq -x_S-(S-x_S)=-S \subseteq B.
$$
Thus, $\{x_S,-x_S\} \subseteq B+v$ for all $v \in A$. This yields $A \subseteq \cc(\{x_S,-x_S\},B)$ and in turn $\dim\left(\cc(\{x_S,-x_S\},B)\right) \ge \dim(A) > k$.
\end{proof}
The equivalence of \ref{dim-k-symmetric-body} and \ref{existence-dim-k-symmetric} is also a consequence of \cite[Corollary~2.8 and Proposition~2.1]{VeenaSangeethaVe2018}. A localized version of that equivalence is \cite[Theorem~2.7]{VeenaSangeethaVe2018}.
\begin{remark}
The dependence of conditions \ref{nonunique-symmetric-body}, \ref{nonunique-symmetric-2}, \ref{existence-segment} from \hyperref[thm:UniqueCircumSymmetric]{Theorem~\ref*{thm:UniqueCircumSymmetric}} in the case of generalized Minkowski spaces $(\R^n,\gamma)$ with unit ball $B$ is as follows:
\begin{itemize}[leftmargin=*,align=left,noitemsep]
\item{\ref{nonunique-symmetric-2} implies \ref{nonunique-symmetric-body}. If $n=2$ then \ref{nonunique-symmetric-body} implies \ref{nonunique-symmetric-2}, because \ref{nonunique-symmetric-2} coincides with \hyperref[dim-k-asymmetric-n]{(b$_0^\prime$)}. If $n \ge 3$ then \ref{nonunique-symmetric-body} does not imply \ref{nonunique-symmetric-2}, see \hyperref[Ex:Sharpness_of_new_theorem]{Example~\ref*{Ex:Sharpness_of_new_theorem}}.}
\item{\ref{nonunique-symmetric-body} implies \ref{existence-segment}, because \hyperref[existence-dim-k-asymmetric]{(c$_0^\prime$)} implies \ref{existence-segment}. So, if some $K \in \mathcal{K}^n$ has several circumcenters then the unit ball is not strictly convex.}
\item{\ref{existence-segment} does not imply \ref{nonunique-symmetric-body}: An $n$-dimensional simplex is a striking example of a unit ball that is not strictly convex, but every $K \in \mathcal{K}^n$ has a unique circumcenter. The unique circumball of $K$ is the intersection of the $n+1$ half-spaces having the same outer normals as the facets of the simplex and whose bounding hyperplanes support $K$.}
\end{itemize}
\end{remark}

Finally, let us point out that condition \hyperref[existence-dim-k-asymmetric]{(c$_0^\prime$)} gives rise to the following criterion in generalized Minkowski planes (cf.~\cite[Lemma~1]{Vai}).

\begin{corollary}
In a generalized Minkowski plane $(\R^2,\gamma)$ every $K \in \mathcal K^2$ has a unique circumcenter if and only if the unit ball does not contain opposing parallel segments in its boundary.
\end{corollary}


\providecommand{\bysame}{\leavevmode\hbox to3em{\hrulefill}\thinspace}
\providecommand{\href}[2]{#2}

\end{document}